\documentclass[11pt]{article}
\usepackage[T1]{fontenc}
\usepackage{amsfonts}
\usepackage{amsmath}
\usepackage{amssymb}
\usepackage{amsthm}
\usepackage{bbm}
\usepackage{bm}
\usepackage{mathrsfs}
\usepackage{verbatim}
\usepackage{setspace}
\usepackage{color}
\usepackage{pdfsync}
\usepackage{enumitem}

\theoremstyle{plain}
\newtheorem{theorem}{Theorem}[section]
\newtheorem{proposition}[theorem]{Proposition}
\newtheorem{lemma}[theorem]{Lemma}
\newtheorem{corollary}[theorem]{Corollary}

\theoremstyle{definition}

\newtheorem{remark}[theorem]{Remark}

\theoremstyle{remark}

\renewenvironment{thebibliography}[1]{%
\begin{oldthebibliography}{#1}%
\setlength{\baselineskip}{1em}
\linespread{1}
\small
\setlength{\parskip}{0.5ex}%
\setlength{\itemsep}{.5em}%
}%
{%
\end{oldthebibliography}%
}

\newcommand{\eps}{\varepsilon}

\newcommand{\N}{\mathbb{N}}

\newcommand{\R}{\mathbb{R}}

\newcommand{\X}{\mathbf{X}}
\newcommand{\bX}{\mathbf{X}}
\newcommand{\Y}{\mathbf{Y}}
\newcommand{\Z}{\mathbf{Z}}
\newcommand{\bZ}{\mathbf{Z}}

\newcommand{\cB}{\mathcal{B}}

\newcommand{\cE}{\mathcal{E}}
\newcommand{\cF}{\mathcal{F}}

\newcommand{\cM}{\mathcal{M}}

\newcommand{\bH}{\mathbf{H}}

\DeclareMathOperator{\supp}{supp}

\DeclareMathOperator{\id}{Id}

\newcommand{\br}[1]{\langle #1 \rangle}

\numberwithin{equation}{section}

\usepackage[pdfborder={0 0 0}]{hyperref}
\hypersetup{
  urlcolor = black,
  pdfauthor = {Mathias Beiglböck, Marcel Nutz},
  pdfkeywords = {Martingale inequality; Concave envelope; Fixed point; Robust hedging; Tchakaloff's theorem},
  pdftitle = {Martingale Inequalities and Deterministic Counterparts},
  pdfsubject = {Martingale Inequalities and Deterministic Counterparts},
  pdfpagemode = UseNone
}

\begin{document}

\title{\vspace{-0cm}
Martingale Inequalities and Deterministic Counterparts
\date{\today}
\author{
  Mathias Beiglb\"ock%
  \thanks{
  Department of Mathematics, University of Vienna, mathias.beiglboeck@univie.ac.at. Research supported by FWF Grants P21209 and P26736.
  }
  \and
  Marcel Nutz%
  \thanks{
  Departments of Statistics and Mathematics, Columbia University, mnutz@columbia.edu. Research supported by NSF Grant DMS-1208985. We are greatly indebted to Josef Teichmann for illuminating discussions about Tchakaloff's theorem which led to its martingale version as stated in the text. We would also like to thank Erhan Bayraktar, the Associate Editor and two anonymous referees for their constructive comments. 
  }
 }
}
\maketitle 

\begin{abstract}
We study martingale inequalities from an analytic point of view and show that a general martingale inequality can be reduced to a pair of deterministic inequalities in a small number of variables. More precisely, the optimal bound in the martingale inequality is determined by a fixed point of a simple nonlinear operator involving a concave envelope. Our results yield an explanation for certain inequalities that arise in mathematical finance in the context of robust hedging.
\end{abstract}

\vspace{.9em}

{\small

}

\section{Introduction}\label{se:intro}

Martingale inequalities are abundant in many areas of probability theory and analysis; see e.g.\ Burkholder's survey~\cite{Burkholder.91} for an extensive list of literature. We study general inequalities for discrete-time martingales from a bird's eye view and relate them to certain deterministic inequalities. Indeed, we shall see that every martingale inequality can be obtained as a consequence of two deterministic ones, and in fact that martingale inequalities are not very probabilistic in nature.

A simple example of a martingale inequality is Doob's maximal quadratic inequality, stating that the running maximum $M_T^*:=\sup_{0\leq t\leq T} |M_t|$ of any martingale $M$ satisfies
\[
  \|M_T^*\|_2\leq 2 \|M_T\|_2,
\]
where $\|\,\cdot\,\|_2$ is the $L^2$-norm.
We may cast this in the form $E[f(M_T,M_T^*)]\leq0$ for a suitable function $f$; namely, $f(x,y)=y^2-4|x|^2$.
The general form of the martingale inequality that we shall consider is
\begin{equation}\label{eq:martIneq1}
  E[f(Z_T)]\leq a,
\end{equation}
where $a$ is a constant and $Z=(Z_t)_{t\in\N}$ is a suitable state process defined as a function of $M$; in the preceding example, $Z=(M,M^*)$. More precisely, let $\X$ be a vector space (in which our martingales are taking values) and let $\Z$ be a set, to be used as the state space. Then the $\Z$-valued process $Z$ is determined by a function $\phi: \Z\times \X \to \Z$ via $Z_{t+1}=\phi(Z_t,M_{t+1}-M_t)$ and some initial value $z_0$. Again in the example, $\phi(x,y,d)=(x+d,y\vee|x+d|)$ updates $M$ by adding the next increment and increases the running maximum if necessary.

Given $f:\Z\to \overline{\R}$ and $\phi$, we may ask if there exists a finite constant $a$ such that~\eqref{eq:martIneq1} holds for all $T\in\N$ and all martingales with prescribed initial value, and what the optimal (minimal) value for $a$ is. A possible answer runs as follows. Consider the operator $A$ which acts on functions $g: \Z\to\overline{\R}$ by pre-composing with $\phi$ and taking the concave envelope at the origin in the variable corresponding to the martingale increment:
\[
  Ag(z)= g(\phi(z,\cdot))^\sharp(0),\quad z\in\Z.
\]
If $u$ is a fixed point of $A$ dominating $f$; that is, $Au=u$ and $u\geq f$, then $a=u(z_0)$ is an admissible constant in~\eqref{eq:martIneq1}. Under the natural condition $\phi(z,0)=z$, a simple monotonicity argument shows that $A$ has a \emph{minimal} fixed point $u$ dominating $f$. This fixed point can be obtained from $f$ by iterating $A$ and passing to the limit,
\[
  u=A^\infty f:= \lim_{n\to\infty} A^n f,
\]
and we shall see that $a=u(z_0)$ is the optimal constant in~\eqref{eq:martIneq1}. In this sense, we may say that a martingale inequality can be reduced to the two deterministic inequalities
\[
  u\geq Au\quad\mbox{and} \quad u\geq f.
\]
(Here $u\geq Au$ is actually equivalent to $Au=u$.) In fact, we may note that $u$ defines a stronger martingale inequality altogether. Namely, as $u\geq f$,
\[
  E[u(Z_T)]\leq u(z_0)
\]
is stronger than the original inequality $E[f(Z_T)]\leq u(z_0)$ with optimal constant, and we remark that the inequality $u\geq f$ is strict in most cases of interest. Returning to our example, we can check that the minimal fixed point is given by
\[
  u(x,y)=\begin{cases}
      y^2- 4|x|^2 & \text{if } |x| < y/2,\\
      2y^2-4|x|y & \text{if } |x| \geq y/2,
   \end{cases}
\]
and so the optimal constant corresponding to the initial value $z_0=(x_0,|x_0|)$ is $a=-2|x_0|^2$, while the %
knowledge of $u$ actually yields a further strengthening of Doob's maximal inequality (Corollary~\ref{co:doob}).
There are of course very relevant martingale inequalities which hold only for some specific class of martingales; for instance, nonnegative martingales, martingales with increments bounded by one, etc. Many such inequalities can be fitted within our framework by choosing $\Z$ appropriately and assigning the value $-\infty$ to the function $f$ on a suitable subset (see also Section~\ref{se:subordinate}).

All this has little to do with probability or measure theory; in fact, it seems that the latter is only needed to define the expectations. In order to clearly separate this aspect (and also to spare the reader some measurable selection arguments), we shall develop the theory for simple martingales (i.e., martingales taking finitely many values), so that all expectations are actually finite sums. In most cases of interest, $f$ and $\phi$ (and then also the fixed point~$u$) have some continuity properties and the passage to general martingales can be done a posteriori by approximation. However, we also provide an alternative argument which is more in the spirit of this paper and applies even to functions that are merely measurable, under the  restriction that $\X$ be finite-dimensional. Namely, we devise a martingale version of Tchakaloff's theorem, stating that given a measurable (integrable) function $g:(\R^n)^{T}\to\R$ and an $n$-dimensional martingale $M$, we can find a \emph{simple} martingale $N$ such that
\[
  E[g(N_1,\dots,N_T)]=E[g(M_1,\dots,M_T)],
\]
and moreover the (finite) support of the law of $N$ lies in the support of the law of $M$. Note that we have here an actual equality; no approximation is necessary.

The theory outlined in this paper can be seen as a general formulation of a strategy of proof that was used in several works of D.~L.~Burkholder for martingale inequalities where $Z$ consists of $X$, its running maximum and its square function. 
Namely, he used a class of functions $u$, corresponding roughly to what we call fixed points, to find admissible or sharp constants in various martingale inequalities,
and in fact it seems that he was aware of at least part of the structure presented here; see in particular Theorem~2.1 in~\cite{Burkholder.02} but also~\cite{Burkholder.88, Burkholder.91, Burkholder.97}, among others, as well as the recent monograph and review article of Os{\c{e}}kowski~\cite{Osekowski.12,Osekowski.13}. 

A different stream of literature about martingale inequalities has emerged in mathematical finance, starting with Hobson~\cite{Hobson.98}. In this context, the process~$X$ takes values in $\R^n$ and represents the discounted prices of $n$ tradable securities, while $f(Z_T)$ is seen as an option maturing at the fixed time horizon~$T$. The problem is to find a minimal constant $a$ and a predictable process~$H$ (i.e., $H_t$ is a function of $X_0,\dots, X_{t-1}$) such that
\begin{equation}\label{eq:hedgingIntro}
  a + \sum_{t=1}^T \br{H_t, X_{t}-X_{t-1}} \geq f(Z_T),
\end{equation}
where the inner product $\br{H_t, X_{t}-X_{t-1}}$ is interpreted as the gain or loss that occurs as the price $X_{t-1}$ changes to $X_{t}$ while $H_t$ units of the security are held. Thus, if $a$ is charged as the price of the option, the trading strategy $H$ allows to hedge the risk of $f(Z_T)$ in a robust (model-free) way. We observe that by taking expectations on both sides,~\eqref{eq:hedgingIntro} implies the martingale inequality $E[f(Z_T)]\leq a$.
Along these lines, ``pathwise'' proofs for several martingale inequalities have been obtained. For these and related results in robust finance, see \cite{AcciaioBeiglbockPenknerSchachermayer.12,AcciaioEtAl.12,BeiglbockHenryLaborderePenkner.11,BeiglbockSiorpaes.13,BrownHobsonRogers.01,CoxObloj.11,DolinskySoner.12,OblojYor.06}
among others; more references can be found in the surveys by Hobson~\cite{Hobson.11} and Ob{\l}{\'o}j~\cite{Obloj.04}.
In particular, a result of Bouchard and Nutz~\cite{BouchardNutz.13} implies that any martingale inequality in finite discrete time can be related to an inequality of the type~\eqref{eq:hedgingIntro}. However, the machinery used there (to deal with a more general case) only yields a non-constructive existence result for~$H$ and little insight into the nature of the inequality. We shall see that, in essence, $H$ is determined quite explicitly as the derivative of $u(\phi(z,\cdot))^\sharp$.

The remainder of this article is organized as follows. In Section~\ref{se:concaveEnvelopes} we consider martingale inequalities with a fixed time horizon $T$ and relate the optimal constant to certain concave envelopes by dynamic programming. Section~\ref{se:timehom} focuses on martingale inequalities that do not depend explicitly on the time horizon $T$; this further condition of time-homogeneity leads to the fixed point considerations mentioned above. The connection to mathematical finance is also discussed here. In Section~\ref{se:examples}, we illustrate the theory by two simple examples, Doob's maximal $L^p$-inequality and Burkholder's inequality for differentially subordinate martingales. Section~\ref{se:tchakaloff} concludes with the martingale version of Tchakaloff's theorem.

\section{Martingale Inequalities and Concave Envelopes}\label{se:concaveEnvelopes}

It will be convenient to work with functions taking values in the extended real line $\overline{\R}=[-\infty,\infty]$. The convention
\begin{equation}\label{eq:infinityConvention}
  \infty-\infty = -\infty
\end{equation}
is used throughout; in particular, in the definitions of concavity and integrals.
Let $\X$ be a real\footnote{The general case is no more difficult than $\X=\R$. Moreover, most of what follows applies to the complex case without change.} vector space. Given a function $g: \X \to\overline{\R}$, we define its concave envelope $g^\sharp: \X \to\overline{\R}$ as the smallest concave function dominating~$g$, or
\[
  g^\sharp(x)=\inf \{\psi(x)|\, \psi: \X \to\overline{\R}\mbox{ is concave and }\psi\geq g\},\quad x\in\X.
\]
We shall need to take consecutive envelopes over several variables.
Given an integer $t\geq0$ and $g: \X^{t+2}\to\overline{\R}$, we first introduce the function
\[
  g^{\sharp_t}: \X^{t+1}\to\overline{\R},\quad g^{\sharp_t}(x_0,\dots,x_t):= g(x_0,\dots,x_t,\cdot)^\sharp (x_t);
\]
in other words, we pass to the concave envelope in the ultimate variable and evaluate the resulting function at the penultimate variable. For an integer $T\geq0$, we can then define the composition
\[
  \sharp(T)=\sharp_0\circ \cdots \circ \sharp_{T-1}
\]
which maps functions of $T+1$ variables into functions of one variable.

Our first aim is to identify, for a fixed time horizon $T$, the optimal constant for a martingale inequality defined by $f: \X^{T+1}\to\overline{\R}$ in terms of the consecutive envelope $f^{\sharp(T)}$. Given $x_0\in\X$, we shall denote by $\cM^T(x_0)$ the set of all laws of $\X$-valued simple martingales $M_0,\dots, M_T$ satisfying $M_0=x_0$. Note that any expectation $E[f(M_0,\dots,M_T)]$ on the original probability space of the martingale $M$ can be expressed as the expectation $\mu[f]:=E_\mu[f]:=\int f\,d\mu$ of $f$ under the law $\mu$ of $M$; the latter point of view will be more convenient in the sequel. We emphasize that the integral under $\mu\in\cM^T(x_0)$ is a finite sum and therefore does not require any measurability conditions, and moreover that, according to~\eqref{eq:infinityConvention}, we have $\mu[f]=-\infty$ if $\mu[f^+]=\mu[f^-]=\infty$.

\begin{proposition}\label{pr:dualityPathdep}
  Let $f: \X^{T+1}\to\overline{\R}$. Then
  \begin{equation}\label{eq:dppIneq}
    f^{\sharp(T)}(x_0) = \sup_{\mu\in\cM^T(x_0)} \mu[f], \quad x_0\in\X.
  \end{equation}
\end{proposition}

Or, to state the same in different words: proving that an inequality $E[f(M_0,\dots,M_T)]\leq a$ holds for all martingales $M$ starting at $x_0$ boils down to checking that $f^{\sharp(T)}(x_0)\leq a$, and in fact $f^{\sharp(T)}(x_0)$ is the optimal constant.

As a first step towards the proof, we consider the case $T=1$. Noting that $\cM(x)$ is simply the set of all probability measures $\mu$ on $\X$ having finite support and barycenter $\mu[\id_{\X}]=x$, the following identity is essentially classical (see Kemperman~\cite{Kemperman.68}); we state the details for the sake of completeness.

\begin{lemma}\label{le:SupIsConcaveEnvelope}
  Let $g: \X\to\overline{\R}$. Then
  \[
    \sup_{\mu\in\cM(x)} \mu[g]=g^\sharp(x),\quad x\in\X.
  \]
\end{lemma}

\begin{proof}
  Let $x\in\X$ and $\mu\in\cM(x)$; then $\mu$ is a convex combination of Dirac measures, $\mu=\sum_{i=1}^n \lambda_i \delta_{x_i}$, with $\sum \lambda_i x_i= x$. In particular,
  \[
    \mu[g] \leq \mu[g^\sharp] = \sum \lambda_i g^\sharp(x_i) \leq g^\sharp(x)
  \]
  as $g^\sharp$ is concave, showing that $\sup_{\mu\in\cM(x)} \mu[g]\leq g^\sharp(x)$. To see the converse inequality, let $x_1,x_2\in\X$ and $\lambda\in (0,1)$. Given $\eps>0$, there are $\mu^\eps_i\in\cM(x_i)$ such that (with $a\wedge b :=\min\{a,b\}$)
  \[
    \mu^\eps_i[g] \geq \eps^{-1} \wedge \sup_{\mu\in\cM(x_i)} \mu[g] -\eps.
  \]
  Using the fact that $\lambda \mu^\eps_1 + (1-\lambda)\mu^\eps_2\in \cM(\lambda x_1 + (1-\lambda)x_2)$, we then have
  \begin{align*}
    \lambda \sup_{\mu\in\cM(x_1)} \mu[g] + (1-\lambda) \sup_{\mu\in\cM(x_2)} \mu[g]
    & \leq \limsup_{\eps\to0} \lambda \mu^\eps_1[g] + (1-\lambda) \mu^\eps_2[g] \\
    & \leq \sup_{\mu\in\cM(\lambda x_1 + (1-\lambda)x_2)} \mu[g],
  \end{align*}
  showing that $x\mapsto \sup_{\mu\in\cM(x)} \mu[g]$ is concave. In view of $\sup_{\mu\in\cM(x)} \mu[g]\geq \delta_x[g]=g(x)$, the definition of $g^\sharp(x)$ now yields $\sup_{\mu\in\cM(x)} \mu[g]\geq g^\sharp(x)$.
\end{proof}

The extension to the case of a general horizon $T$ can be understood as a dynamic programming argument where the martingale laws play the role of the controls in a stochastic control problem. Given $g: \X^{t+2}\to\overline{\R}$, we therefore introduce the (value) function
\[
  \cE_t(g): \X^{t+1}\to\overline{\R},\quad \cE_t(g)(x_0,\dots,x_t):= \sup_{\mu\in\cM(x_t)} \mu[g(x_0,\dots,x_t,\cdot)],
\]
as well as the composition
\[
  \cE^{t}:=\cE_t\circ \cdots \circ \cE_{T-1}
\]
which maps functions of $T+1$ variables into functions of $t+1$ variables.

\begin{lemma}\label{le:dpp}
  Let $f: \X^{T+1}\to\overline{\R}$. Then
  \begin{equation}\label{eq:dppIneq}
    (\cE_0\circ \cdots \circ \cE_{T-1})(f)(x_0) = \sup_{\mu\in\cM^T(x_0)} \mu[f], \quad x_0\in\X.
  \end{equation}
\end{lemma}

\begin{proof}
  We first suppose that $f$ is bounded from above. To see the inequality ``$\leq$'', let $\eps>0$. For all $0\leq t< T$ and $(x_0,\dots,x_t)\in\X^{t+1}$, let $\mu_t(x_0,\dots,x_t)\in\cM(x_t)$ be such that
  \[
    \mu_t(x_0,\dots,x_t) [\cE^{t+1}(f)(x_0,\dots,x_t,\cdot)] \geq \sup_{\mu\in \cM(x_t)} \mu[\cE^{t+1}(f)(x_0,\dots,x_t,\cdot)] - \eps.
  \]
  We may see $\mu_t$ as a stochastic kernel on $\X^{t+1}$ equipped with the discrete $\sigma$-field. Recalling that we are only using measures with finite support, we may form the product measure $\mu^\eps:= (\mu_0\otimes \cdots\otimes \mu_{T-1})(x_0)$ which is an element of $\cM^T(x_0)$ by Fubini's theorem. We then have
  \[
    (\cE_0\circ \cdots \circ \cE_{T-1})(f)(x_0) \leq \eps T + \mu^\eps[f] \leq \eps T + \sup_{\mu\in\cM^T(x_0)} \mu[f].
  \]
  As $\eps>0$ was arbitrary, this yields the claimed inequality. To see the converse inequality ``$\geq$'', fix $x_0\in\X$ and note that any $\mu\in \cM^T(x_0)$ can be decomposed into the product $\mu=\mu_0\otimes \mu_1\otimes \cdots  \otimes \mu_{T-1}$ of a measure $\mu_0\in\cM(x_0)$ and kernels $\mu_t$ on $\X^t$ such that $\mu_t(x_1,\dots,x_t)\in \cM(x_t)$ for all $x_1,\dots,x_t\in \X$. By the definition of the operators $\cE_t$, we then have
  \[
    (\cE_0\circ \cdots \circ \cE_{T-1})(f)(x_0) \geq (\mu_0\otimes \mu_1\otimes \cdots  \otimes \mu_{T-1})[f]= \mu[f]
  \]
  and the claim follows as $\mu\in\cM^T(x_0)$ was arbitrary.

  Finally, for the case of a general function $f$, we observe that both sides of~\eqref{eq:dppIneq} are continuous along increasing sequences $(f_n)$ of
  $\overline{\R}$-valued functions having the property that $\{f_n=-\infty\}=\{f_{n+1}=-\infty\}$, $n\geq1$.
  Thus, we may apply the above to $f\wedge n$ and pass to the limit as $n\to\infty$.
\end{proof}

\begin{proof}[Proof of Proposition~\ref{pr:dualityPathdep}]
  Since Lemma~\ref{le:SupIsConcaveEnvelope} shows that $\sharp_t=\cE_t$, Proposition~\ref{pr:dualityPathdep} is a direct consequence of Lemma~\ref{le:dpp}.
\end{proof}

\section{Time-Homogeneous Martingale Inequalities}\label{se:timehom}

Let $x_0\in\X$, set $X_0=x_0$ and let $(X_t)_{t=1,2,\dots}$ be the coordinate-mapping process on $\X\times\X\times \cdots$.
Moreover, let $\Z$ be a nonempty set and fix a function $\phi: \Z\times \X \to \Z$. Given $z_0\in\Z$, we define the $\Z$-valued process $Z=(Z_t)_{t=0,1,\dots}$ by
\[
  Z_0=z_0,\quad Z_{t+1}=\phi(Z_t,X_{t+1}-X_t).
\]
We write $\overline{\R}^{\Z}$ for the set of all functions $\Z\to\overline{\R}$, equipped with the pointwise partial order and convergence,
and define the operator $A:\overline{\R}^{\Z}\to\overline{\R}^{\Z}$ via
\[
  Ag(z):= [g\circ \phi(z,\cdot)]^\sharp(0),\quad z\in\Z.
\]
Moreover, we write $A^T$ for the $T$-fold composition $A\circ \cdots\circ A$. Using this notation, Proposition~\ref{pr:dualityPathdep} can be rephrased as follows.

\begin{lemma}\label{le:dualityMarkov}
  Let $f: \Z\to\overline{\R}$ and let $(x_0,z_0)\in \X\times\Z$.  Then
  \[
    A^Tf(z_0) = \sup_{\mu\in\cM^T(x_0)} \mu[f(Z_T)].
  \]
\end{lemma}

This lemma may look less general than Proposition~\ref{pr:dualityPathdep}, which allows for a general dependence on the path of $X$, but let us mention that with the choice $\Z=\N\times \X^\N$ we can arrange things so that $Z_t=(t,X_0,X_1,\dots,X_t,0,0,\dots)$.

From now on, we focus on martingale inequalities which hold for any time horizon~$T$. The structural condition
\begin{equation}\label{eq:freeze}
  \phi(z,0)=z,\quad z\in\Z
\end{equation}
seems to be natural in that setting and we make this a \emph{standing assumption}. The operator $A$ then has the following monotonicity properties.

\begin{lemma}\label{le:Amonotone}
  Let $g,g': \Z\to\overline{\R}$. Then
  \begin{enumerate}
    \item $Ag\geq g$;
    \item $g\geq g'$ implies $Ag\geq Ag'$.
  \end{enumerate}
\end{lemma}

\begin{proof}
  In view of~\eqref{eq:freeze}, we have
  \[
    Ag(z) = g(\phi(z,\cdot))^\sharp(0) \geq g(\phi(z,0))=g(z),\quad z\in\Z.
  \]
  The second property follows from the monotonicity of $\sharp$.
\end{proof}

\begin{theorem}\label{th:minimalFP}
  Let $f:\Z\to\overline{\R}$. Then the limit
  \[
    A^\infty f(z) := \lim_{n\to\infty} A^nf(z),\quad z\in\Z
  \]
  exists in $\overline{\R}$ and the function $A^\infty f\in\overline{\R}^{\Z}$ is characterized as the smallest fixed point of $A$ which dominates $f$.
\end{theorem}

\begin{remark}\label{rk:minimalFPprob}
  By Lemma~\ref{le:dualityMarkov}, $A^\infty f(z_0)$ is the optimal \emph{horizon-independent} constant for the martingale inequality determined by $f$, $\phi$ and $z_0$. In fact, Lemma~\ref{le:dualityMarkov} naturally extends to
  \[
    A^\infty f(z_0) = \sup_{\mu\in\cM^\infty(x_0)} \mu[f(Z_\infty)]
  \]
  if we denote by $\cM^\infty(x_0)$ the set of all laws of $\X$-valued simple\footnote{``Simple'' means that the support is a finite subset of $\X^{\N}$.} martingales $(M_t)_{t\in \N}$ satisfying $M_0=x_0$. Note that any such martingale is eventually constant, so that $Z_\infty:=\lim_n Z_n$ is well-defined $\mu$-a.s.\ for all $\mu\in\cM^\infty(x_0)$.
\end{remark}

\begin{proof}[Proof of Theorem~\ref{th:minimalFP}]
  It follows from Lemma~\ref{le:Amonotone} that
  \[
    f\leq Af \leq \cdots\leq A^nf,\quad n\geq1.
  \]
  In particular, the limit $A^\infty f(z) := \lim_{n\to\infty} A^nf(z) \in \overline{\R}$ exists for all $z\in \Z$. Next, let us observe that if $(g_n)_{n\geq1}\subseteq \overline{\R}^{\Z}$ is a nondecreasing sequence, then
  \[
    \lim_n  g_n^\sharp = (\lim_n g_n)^\sharp.
  \]
  Indeed, both limits are increasing and thus well-defined, and the monotonicity of $\sharp$ immediately implies that $\lim_n  g_n^\sharp \leq (\lim g_n)^\sharp$. Conversely, $\lim_n  g_n^\sharp$ is concave as the pointwise limit of concave functions and dominates $\lim g_n$, so that $\lim_n  g_n^\sharp \geq (\lim g_n)^\sharp$. Using this continuity property of $\sharp$, we see that
  \begin{gather*}
    A^\infty f(z) = \lim_n A^{n+1}f(z) = \lim_n [A^nf\circ \phi(z,\cdot)]^\sharp(0)= [\lim_n A^nf\circ \phi(z,\cdot)]^\sharp(0) \\= [A^\infty f\circ \phi(z,\cdot)]^\sharp(0)=AA^\infty f(z)
  \end{gather*}
  for all $z\in\Z$; that is, $A^\infty f$ is a fixed point. If $g\in\overline{\R}^{\Z}$ is another fixed point of $A$ such that $g\geq f$, then the monotonicity of $A$ from Lemma~\ref{le:Amonotone} yields that
  \[
    g = A^n g \geq A^n f,\quad n\geq0
  \]
  and hence $g\geq A^\infty f$ by passing to the limit.
\end{proof}

\begin{remark}\label{rk:Burkholder}
  Let $u: \Z\to\overline{\R}$ be any function such that $f\leq u$ and $Au\leq u$ (hence $Au=u$; cf.\ Lemma~\ref{le:Amonotone}). Then
  \[
    \sup_{\mu\in\cM^\infty(x_0)} \mu[f(Z_\infty)] =A^\infty f(z_0) \leq A^\infty u(z_0) = u(z_0);
  \]
 that is, to prove that the martingale inequality holds with right-hand side~$a$, it suffices to exhibit a fixed point $u$ of $A$ which dominates $f$ and satisfies $u(z_0)\leq a$. As mentioned in the Introduction, this corresponds to a general formulation of the strategy of proof that has been used by Burkholder for several specific martingale inequalities.
 For the above conclusion, it is not necessary to establish that $u$ is the minimal fixed point; however, this property guarantees that $u(z_0)$ is the optimal right-hand side.
\end{remark}

To find an explicit formula for $A^\infty f$ (or any other fixed point), it is often useful to study properties of $f$ that are preserved by $A$. We give a simple example to illustrate this point (see also Section~\ref{se:DoobIneq}).

\begin{remark}\label{rk:scaling}
  Suppose that $\Z$ is a cone and that $\phi$ is positively homogeneous of degree one. If $f\in\overline{\R}^{\Z}$ is positively homogeneous of degree $p>0$, then so are $Af$ and $A^\infty f$.
  Indeed, let $\lambda\geq 0$; then
    \begin{multline*}
    Af(\lambda z)
    = f(\phi(\lambda z,\cdot))^\sharp(0)
    = \inf \{\psi(0)|\, \psi(\lambda \cdot )\geq f(\phi(\lambda z,\lambda \cdot)) \} \\
    = \inf \{\psi(0)|\, \psi \geq \lambda^p f(\phi(z, \cdot)) \}
    = \inf \{\lambda^p\psi(0)|\, \psi \geq f(\phi(z, \cdot)) \}
    = \lambda^p Af(z),
  \end{multline*}
  where the infima are taken over all concave functions $\psi: \X \to\overline{\R}$. The homogeneity of $A^\infty f$ follows.
\end{remark}

Next, we would like to explain a connection to certain inequalities which have arisen in mathematical finance---from our abstract point of view, we shall see that the latter are simply manifestations of the concavity that is imposed by $A$.
For the purpose of the subsequent discussion, we assume that we are given a dual pair $\X,\X'$ with a separating pairing $\br{\cdot,\cdot}$. Given a concave function $h: \X\to\overline{\R}$, the supergradient $\partial h(d_0)$ at $d_0\in\X$ is defined as the set of all $\xi\in\X'$ such that $h(d_0)+\br{\xi,d-d_0}\geq h(d)$ for all $d\in\X$, and $h$ is called superdifferentiable at $d_0$ if this set is nonempty.

\begin{lemma}\label{le:FPcharacterizations}
  Let $g: \Z\to\overline{\R}$. Each of the following conditions implies the subsequent one:
  \begin{enumerate}
    \item For all $z\in\Z$ there exists $\xi(z)\in\X'$ such that
     \begin{equation}\label{eq:gradientIneq}
       g(\phi(z,d))\leq g(z)+ \br{\xi(z),d},\quad d\in\X.
     \end{equation}
    \item $Ag=g$.
    \item For all $z\in\Z$ and all $\xi(z)\in \partial g(\phi(z,\cdot))^\sharp(0)$,
     \[
       g(\phi(z,d))\leq g(z)+ \br{\xi(z),d},\quad d\in\X.
     \]
  \end{enumerate}
  If the concave function $g(\phi(z,\cdot))^\sharp$ is superdifferentiable at $d=0$, these conditions are equivalent. In particular, the conditions are equivalent if $\X$ is finite-dimensional and $g(\phi(z,\cdot))^\sharp$ is finite-valued.
\end{lemma}

\begin{proof}
  Let (i) hold. Taking concave envelopes on both sides of~\eqref{eq:gradientIneq}, we see that
  \[
    Ag(z)=g(\phi(z,\cdot))^\sharp (0)\leq g(z)+ \br{\xi(z),0} = g(z),
  \]
  which implies~(ii) by Lemma~\ref{le:Amonotone}.
  Let $\xi(z)\in \partial g(\phi(z,\cdot))^\sharp(0)$; that is,
  \[
    g(\phi(z,\cdot))^\sharp(d) \leq g(\phi(z,\cdot))^\sharp(0) + \br{\xi(z),d} \equiv Ag(z) + \br{\xi(z),d},\quad d\in\X.
  \]
  Then~(ii) and the fact that $g(\phi(z,d))\leq g(\phi(z,\cdot))^\sharp(d)$ yield~(iii).
  Finally, if $\partial g(\phi(z,\cdot))^\sharp(0)\neq\emptyset$ for all $z\in\Z$, it is evident that~(iii) implies~(i).
\end{proof}

We mention that Lemma~\ref{le:FPcharacterizations} can serve as a tool to verify that $g$ is a fixed point: in examples, it is sometimes easier to verify a relation like~\eqref{eq:gradientIneq} which does not involve the concave envelope (e.g.\ \cite{BeiglbockSiorpaes.13}).

\begin{remark}\label{rk:mathFin}
  In the context of mathematical finance, the $\R^n$-valued process $X$ represents the discounted prices of $n$ tradable securities, while $f(Z_T)$ is seen as an option maturing at time~$T$. Inequality~\eqref{eq:gradientIneq} with $g=u=A^\infty f$ expresses that the trading strategy $H_t:=\xi(Z_{t-1})$ yields a superhedge for the seller of the option if $u(z_0)$ is charged as its price:
  \begin{equation}\label{eq:hedging}
    u(z_0) + \sum_{t=1}^T \br{H_t, X_{t}-X_{t-1}} \geq u(Z_T) \geq f(Z_T),
  \end{equation}
  where the left-hand side is the balance obtained from the amount $u(z_0)$ and the gains/losses from trading according to $H$. A similar observation applies if the time horizon $T$ is seen as fixed (which is more natural in finance); namely, $H_t \in \partial [A^{T-t}(\phi(Z_{t-1},\cdot))^\sharp](0)$ yields a process such that
  \begin{equation}\label{eq:hedging2}
    A^Tf(z_0) + \sum_{t=1}^T \br{H_t, X_{t}-X_{t-1}} \geq f(Z_T).
  \end{equation}
  In particular, this gives a simple and constructive proof for the result of~\cite{BouchardNutz.13} mentioned in the Introduction (note that an element of the supergradient can be chosen simply by taking a directional derivative).

  By its definition, $A^T(z_0)$ is the minimal constant allowing for an inequality of the form~\eqref{eq:hedging2} to hold almost-surely under all martingale laws and hence in all viable models, so that $A^T(z_0)$ is called the robust (or model-independent) superhedging price.
  To enlarge a bit further on the financial aspect, suppose that $\Z\subseteq \X\times\Y$ for some set $\Y$ and that $\phi(x,y,d)=\varphi(x+d,y)$ for some function $\varphi: \X\times \Y\to \Z$, where we now write $(x,y)$ instead of $z$ (see also Section~\ref{se:DoobIneq} below). If $u=A^\infty f$, then $u(\cdot,y)$ is concave because $u(x,y)$ is the concave envelope of $u(\varphi(\cdot,y))$ evaluated at $x$, and moreover
  \[
    \partial_x u(x,y) = \partial u(\phi(x,y,\cdot))^\sharp(0).
  \]
  In other words, the hedging strategy is given by $\xi(x,y)=\partial_x u(x,y)$, which corresponds to the option's Delta in the language of finance.
\end{remark}

Certain classical martingale inequalities hold also for submartingales. This can be related to the above as follows (the submartingale property is understood componentwise in the multivariate case).

\begin{remark}
  Let $\X=\R^n$ and $g: \X\to\overline{\R}$; then by Lemma~\ref{le:SupIsConcaveEnvelope}, we have $\sup_{\mu\in\cM(x)} \mu[g]=g^\sharp(x)$. Now let $\cM^*(x)$ be the set of all probability measures on $\X$ having finite support and barycenter $x^*\geq x$. If the function $g$ is (componentwise) nonincreasing, we also have
  \[
    \sup_{\mu\in\cM^*(x)} \mu[g]=g^\sharp(x),\quad x\in\X.
  \]
  Indeed, for each $\mu^*\in\cM^*(x)$ there is $\mu\in\cM(x)$ such that
  $\mu^*[g]\leq \mu[g]$. As a consequence, the martingale inequality corresponding to $f$ and $\phi$ extends to submartingales under the condition that
  \[
    A^\infty f(\phi(z,\cdot))\quad\text{is nonincreasing.}
  \]
  Some martingale inequalities extend only to, e.g., nonnegative submartingales. Such a case can be covered by choosing a suitable state space $\Z$, as in Section~\ref{se:subordinate} below.
\end{remark}

We conclude this section with a brief remark about measurability questions (which we have avoided wherever possible).

\begin{remark}
  Suppose that $\X=\R^n$ and $\Z$ is, say, a Polish space, and that $\phi$ is Borel-measurable. If $f$ is Borel-measurable, one can check that $Af$ and $A^\infty f$ are upper-semianalytic and in particular universally measurable; however, it can happen that $Af$ is not Borel-measurable. As a consequence, the hedging strategy in Remark~\ref{rk:mathFin} can also be chosen to be universally measurable.
\end{remark}

\section{Examples}\label{se:examples}

\subsection{Doob's Maximal Inequality}\label{se:DoobIneq}

The aim of this subsection is to illustrate the above abstract theory by a ramification of Doob's maximal $L^p$-inequality; in this case, all quantities of interest can be computed explicitly. In what follows, $\X$ is a vector space with norm~$|\cdot|$.

\begin{proposition}\label{pr:Doob}
Let $1<p<\infty$, $\bZ=\{(x,y)\in \bX\times\R_+:|x|\leq y\}$ and
\[
  \phi(x,y,d)=(x+d, y\vee|x+d|), \quad f(x,y)=y^p- (\tfrac p{p-1})^p|x|^p, \quad (x,y,d)\in \bZ\times\bX.
\]
Then the minimal fixed point of $A$ dominating $f$ is given by
\begin{align}\label{DoobFixPoint}
  A^\infty f(x,y)=\begin{cases}
  f(x,y) & \text{if } |x| < \tfrac{p-1}p y,\\
  \tilde u(x,y) & \text{if } |x| \geq \tfrac{p-1}p y,
\end{cases}
\end{align}
where
\[
  \tilde u(x,y):= py^p- \tfrac{p^2}{p-1}|x|y^{p-1}, \quad (x,y)\in \bZ.
\]
\end{proposition}

\begin{remark}\label{rk:DoobWithC}
  The proof below also shows that the constant $(\tfrac p{p-1})^p$ in the definition of $f$ is optimal. Namely, if
  \begin{equation}\label{eq:DoobWithC}
    f_c(x,y)=y^p- c|x|^p
  \end{equation}
  for $c\geq0$, we shall see that $A^\infty f_c\equiv \infty$ for $c<(\tfrac p{p-1})^p$, whereas $A^\infty f_c$ is finite-valued for $c\geq (\tfrac p{p-1})^p$.
\end{remark}

Setting $|M|^*_T= \max_{0\leq t\leq T} |M_t|$ and applying the results of the previous subsection, we immediately deduce the following ramification of Doob's maximal $L^p$-inequality.

\begin{corollary}\label{co:doob}
  For all $(x,y)\in \Z$, $T\geq0$ and every (simple) $\X$-valued martingale $M$ starting at $M_0=x$, we have
  \[
    E\big[(|M|^*_T)^p\vee y^p
     - (\tfrac{p}{p-1})^p|M_T|^p\big]\leq
   \begin{cases}
      y^p- (\tfrac p{p-1})^p|x|^p & \text{if } |x| < \tfrac{p-1}p y,\\
      py^p-\tfrac{p^2}{p-1}|x|y^{p-1} & \text{if } |x| \geq \tfrac{p-1}p y
   \end{cases}
  \]
  and the right-hand side is optimal.
  In particular, for the case $y=|x|$, we have
  \begin{equation}\label{eq:DoobDiagonal}
    E\big[(|M|^*_T)^p - (\tfrac{p}{p-1})^p|M_T|^p\big]\leq -\tfrac{p}{p-1}|x|^p \leq 0
  \end{equation}
  and thus $\||M|^*_T\|_p \leq \tfrac{p}{p-1}\|M_T\|_p$.
\end{corollary}

We mention that the function $\tilde{u}$ also appears in a proof of~\eqref{eq:DoobDiagonal} in~\cite{Burkholder.91}. The optimality of the constant was not studied there; incidentally, we see that $\tilde{u}(x,y)$ actually yields the optimal constant for initial conditions with $y=|x|$. The function $\tilde{u}$ can also be extracted (with some additional work) from Cox~\cite{Cox.84}, who considers the finite-horizon version of Doob's inequality in the case $\X=\R$.

\begin{proof}[Proof of Proposition~\ref{pr:Doob} and Remark~\ref{rk:DoobWithC}]
  Fix $c\geq 0$ and let $f:=f_c$ be defined as in~\eqref{eq:DoobWithC}. By Remark~\ref{rk:minimalFPprob}, the function $u:= A^\infty f$ has the representation
  \begin{equation}\label{eq:doobValue}
    u(x,y) = \sup_{\mu\in\cM^\infty(x)} \mu[f(Z_\infty)],
  \end{equation}
  and in view of the form of $f$, this implies that $u(x,y)$ depends on $x$ only through $|x|$. Moreover, we have the scaling property $u(\lambda x, \lambda y) = \lambda^p u(x,y)$ for $\lambda \geq 0$; cf.\ Remark~\ref{rk:scaling}. Thus, $u$ is completely described by the function
  \[
    \varrho: [0,1]\to \overline{\R},\quad \varrho(|x|):=u(x,1);
  \]
  namely, we have $u(0,0)=0$ and $u(x,y)= y^p \varrho(|x| / y) $ for all $(x,y)\in \Z$ with $y>0$.
  On the other hand, we know that $u$ is a fixed point of $A$,
  \begin{equation}\label{DoobFixedPoint}
    u(x,y)=A u(x,y) =u(x+\cdot, y\vee|x+ \cdot|)^\#(0) =u(\cdot, y\vee|\cdot|)^\#(x),
  \end{equation}
  so that $x\mapsto u(x,y)$ is concave. In particular, using $u\geq f$ and the scaling property, we see that $u(x,y)=\infty$ at one point $(x,y)$ if and only if $u\equiv \infty$ on $\Z$. For the time being, let us suppose that we are in the case where $u$ is finite.

  Under this condition, it follows from \eqref{DoobFixedPoint} and the scaling property, or also directly from~\eqref{eq:doobValue}, that $x\mapsto u(x,y)$ is continuous. Thus, $\varrho$ is a continuous concave function on $[0,1]$, and it follows from \eqref{DoobFixedPoint} that its (left) tangent~$t$ at the boundary point $r=1$ satisfies
  \begin{align}\label{FitCondition}
    t(r)\geq r^p\varrho(1),\quad r\in [1,\infty);
  \end{align}
  note that $r^p\varrho(1)=u(x_r,|x_r|)$ if $x_r\in\X$ is any point with $|x_r|=r$ (we may assume that $\X\neq\{0\}$). For later use, we remark that the converse is also true: a continuous concave function $\bar{\varrho}$ on $[0,1]$ satisfying the analogue of \eqref{FitCondition} determines a fixed point $\bar{u}$ of $A$.

  Let us establish that
  \begin{equation}\label{eq:doobSigns}
    \varrho(0)\geq1\quad\mbox{and}\quad\varrho(1)<0.
  \end{equation}
  Indeed, $\varrho(0)=u(0,1)\geq f(0,1)=1$. Moreover, if $\varrho(1)$ were nonnegative, then $p>1$ and~\eqref{FitCondition} would imply that the tangent $t$ has nonnegative slope, thus $\varrho(1)=t(1)\geq t(0) \geq \varrho(0) \geq1$. But then \eqref{FitCondition} states that the affine function $t(r)$ dominates $r^p$ on $[1,\infty)$, which is impossible.

  As a result, $r\mapsto r^p\varrho(1)$ is concave and we see that the tangent condition~\eqref{FitCondition} can be stated equivalently in differential terms. Namely, if $\varrho'(1)$ denotes the slope of $t$, \eqref{FitCondition} is equivalent to
  \begin{align}\label{FitCondition2}
    0 > \varrho'(1) \geq p \varrho(1).
  \end{align}
  In view of~\eqref{eq:doobSigns}, the tangent $t$ has a unique zero $r_1$ in  $[0,1]$. Using the Intercept Theorem, \eqref{FitCondition2} implies that
  \[
   \frac{1-r_1}{\varrho(1)} = \frac{1}{\varrho(1)-t(0)} = \frac{1}{\varrho'(1)}\leq \frac{1}{p\varrho(1)}
  \]
  and hence
  \begin{align}\label{FitCondition3}
  r_1\leq 1-1/p.
  \end{align}

  Next, we construct another fixed point of $A$ for comparison. Let $\bar t$ be the (uniquely determined) affine function which is parallel to $t$ and touches
  \[
    r \mapsto f(x_r,1), \quad r\in [0,1].
  \]
  We denote by $(r_2,f(x_{r_2},1))$ the coordinates of this touching point. Set
  \begin{align}\label{BetterAlpha}\bar\varrho(r):= \begin{cases}
  f(x_r,1) & \mbox{for } r\in [0, r_2],\\
  \bar t(r) & \mbox{for } r\in (r_2,1].
  \end{cases}
  \end{align}
  By definition, $\bar \varrho$ is a continuous concave function satisfying \eqref{FitCondition2}. As remarked above, this implies that $\bar\varrho$ defines a fixed point $\bar{u}$ of $A$ via $\bar u(0,0):=0$ and $\bar u(x,y):= y^p \bar\varrho(|x| / y) $ for $(x,y) \in \bZ$ with $y>0$.

  The fact that $f\leq u$ and the construction of $\bar{u}$ imply that $\bar{u} \leq u$. On the other hand, we have $\bar u \geq f$ and $u$ is the minimal fixed point of $A$ above $f$, so $u\leq \bar{u}$. As a result, $\bar u=u$, $\bar \varrho= \varrho$ and $\bar t= t$. In particular, this establishes that $\varrho$ is of the specific form \eqref{BetterAlpha}; it remains to determine the tangent $t$ explicitly.

  Consider $r_0:=(1/c)^{1/p}$, the zero  of $r\mapsto 1-cr^p=f(x_r,1)$. By concavity, we must have $r_0\leq r_1$; recall that $r_1$ is the zero of the tangent. In view of~\eqref{FitCondition3}, we conclude that
  \begin{equation}\label{eq:doobTouch}
    r_0\leq r_1\leq  1-1/p;
  \end{equation}
  hence, our assumption that $u$ is finite is contradicted whenever $c<(\tfrac p{p-1})^p$.

  Suppose that $c=(\tfrac p{p-1})^p$. Then $r_0 = 1-1/p$ and so~\eqref{eq:doobTouch} implies that $r_1=r_0=1-1/p$. The slope of $r\mapsto f(x_r,1)$ in this point is $-\frac{p^2}{p-1}$; therefore,
  \[
    t(r)= - r\frac{p^2}{p-1}+p.
  \]
  In view of~\eqref{BetterAlpha}, this corresponds to the claimed formula~\eqref{DoobFixPoint}. Since we have seen that this form of $u$ defines a fixed point dominating $f$, we are necessarily in the case where $A^\infty f$ is finite; moreover, as $f$ is decreasing with respect to~$c$, $A^\infty f$ is then also finite for all $c\geq (\tfrac p{p-1})^p$.
\end{proof}

\subsection{Differentially Subordinate Martingales}\label{se:subordinate}

The main purpose of this subsection is to illustrate how one can accommodate a martingale inequality which holds only for a specific class of martingales. To this end, we shall treat an inequality for differentially subordinate martingales, first derived by Burkholder for real-valued processes in~\cite{Burkholder.84} and extended to the Hilbert-valued case in~\cite{Burkholder.88}. A martingale $N$ is differentially subordinate to another martingale $M$ if $|N_{t+1}-N_t|\leq |M_{t+1}-M_t|$ for all $t\geq0$. In other words, this says that the increments of the bivariate martingale $(M,N)$ take values in the cone $\{(d_1,d_2):\, |d_2|\leq |d_1|\}$, and this is the condition defining the class of (bivariate) martingales for which the inequality will hold.

Let $\bH$ be a Hilbert space. In what follows, our basic vector space is $\X:=\bH\times \bH$ and our state space is $\Z=\X \cup\{\Delta\}$; the additional point $\Delta$ will be used as a cemetery state for paths that violate the subordination condition.

\begin{proposition}\label{pr:subordinateIneq}
  Let $1<p<\infty$ and $p^*=\max\{p,p/(p-1)\}$. For $z\in \Z$ and $d=(d_1,d_2)\in \X$, define
  \[
    \phi(z,d)=
    \begin{cases}
      \Delta & \text{if } z=\Delta \text{ or }|d_2|>|d_1|, \\
      z+d & \text{otherwise,}
    \end{cases}
  \]
  \[
    f(z)=
    \begin{cases}
      -\infty & \text{if } z=\Delta, \\
      |x_2|^p-(p^*-1)^p|x_1|^p & \text{if } z=(x_1,x_2)\in \X,
    \end{cases}
  \]
  \[
    \tilde{u}(z)=
    \begin{cases}
      -\infty & \text{if } z=\Delta, \\
      p(1-1/p^*)^{p-1} (|x_2|-(p^*-1)|x_1|) (|x_1|+|x_2|)^{p-1} & \text{if } z\in \X.
    \end{cases}
  \]
  Then the minimal fixed point of $A$ dominating $f$ is given by $A^\infty f=u$, where $u$ is defined for $1<p\leq 2$ by
  \[
    u(z)=
    \begin{cases}
      -\infty & \text{if } z=\Delta, \\
      \tilde{u}(z) & \text{if } z=(x_1,x_2)\in \X \text{ and } |x_2|\leq (p^*-1)|x_1|,\\
      f(z) & \text{if } z=(x_1,x_2)\in \X \text{ and } |x_2|> (p^*-1)|x_1|\\
    \end{cases}
  \]
  and by the same identity with $\tilde{u}$ and $f$ interchanged if $2\leq p<\infty$.
\end{proposition}

\begin{corollary}
  Let $1<p<\infty$ and $p^*=\max\{p,p/(p-1)\}$. Let $M^1,M^2$ be $\bH$-valued (simple) martingales starting at $(M^1_0,M^2_0)=(x_1,x_2)$ and satisfying $|M^2_{t+1}- M^2_t| \leq |M^1_{t+1}- M^1_t|$ for all $t\geq0$. Then
  \[
    E[|M^2_T|^p - (p^*-1)^p |M^1_T|^p] \leq u(x_1,x_2)
  \]
  and in particular $\|M^2_T\|_p \leq (p^*-1)\|M^1_T\|_p$ if $x_1=x_2$.
\end{corollary}

\begin{proof}[Proof of Proposition~\ref{pr:subordinateIneq}]
  All relevant properties are contained in~\cite{Burkholder.88}; we merely translate them into our setup. Indeed, we have $f(\Delta)=u(\Delta)$ by definition, and it is checked below Equation~(1.10) in~\cite{Burkholder.88} that $f(z)\leq \tilde{u}(z)$ for $z\in\X$. Hence, $f\leq u$.
  Moreover, according to Remark~1.2 in~\cite{Burkholder.88}, $u$ is the smallest function which dominates $f$ on $\X$ and has the property that $r\mapsto u(z+rd)$ is concave for all $z\in\X$ and all $d=(d_1,d_2)\in\X$ such that $|d_2|\leq|d_1|$. Using our notation and recalling that $u(\phi(z,\cdot))=-\infty$ outside the set $\{|d_2|\leq|d_1|\}$, it follows that $u$ is the smallest function dominating $f$ on $\Z$ such that $u(\phi(z,\cdot))$ is concave on $\X$. The latter property implies that
  \[
    Au(z)=u(\phi(z,\cdot))^\sharp (0) = u(\phi(z,0))=u(z),
  \]
  so $u$ is a fixed point of $A$. Conversely, if $g: \Z\to\overline{\R}$ is any fixed point of $A$, then $g(\phi(z,d))=g(z+d+\cdot)^\sharp (0)$ and hence $g(\phi(z,\cdot)$ is concave. As a result, $u$ is the smallest fixed point of $A$ dominating $u$.
\end{proof}

\section{Tchakaloff's Theorem for Martingales}\label{se:tchakaloff}

In the preceding sections, we have restricted our attention to simple martingales and we still have to argue that this entails no essential loss of generality. On the one hand, let us mention again that for nice functions $f$ and $\phi$, the extension from simple to general martingales can be done by direct approximation arguments; see, e.g., the proofs of Lemma~2.2 in~\cite{Burkholder.02} or Theorem~2.2 in~\cite{Burkholder.81}. On the other hand, we have developed the theory without regularity conditions and so we would like to see that the extension can be achieved under the natural requirement necessary to define the expectations; namely, the measurability alone. This will be achieved by a martingale version of Tchakaloff's theorem.

Following Bayer and Teichmann~\cite{BayerTeichmann.06}, a general version of Tchakaloff's classical theorem~\cite{Tchakaloff.57} about the existence of cubature formulas can be stated as follows: given an integrable function $f$ on a probability space $(\Omega,\cF,\mu)$, there exists a probability measure $\nu$ with finite support such that $\nu[f]=\mu[f]$, and moreover that support can be chosen to lie in the support of $\mu$. The function $f$ may be multivariate, which allows one to incorporate a finite number of linear constraints on $\nu$; for instance, that $\nu$ should have the same first moment as $\mu$. Our aim is to provide a version of the theorem where $\mu$ and $\nu$ are martingale laws. This extension is not immediate because the martingale property corresponds to an infinite number of constraints\footnote{We thank Josef Teichmann for the insightful discussions which led to this theorem.}.


\begin{theorem}\label{th:tchakaloff}
  Let $k,n,T\in \N$ and $\X=\R^n$. Let $x_0\in\X$ and let $\mu$ be the law of an $\X$-valued martingale $M_0,\dots, M_T$ with $M_0=x_0$, and let $A\subseteq \X^{T+1}$ be a ($\mu$-measurable) set such that $\mu(A)=1$. Moreover, let $f: \X^{T+1}\to \R^k$ be a $\mu$-measurable function such that $\mu[|f|]<\infty$. There exists a martingale law $\nu$, still starting at $x_0$, such that $\#\supp \nu\leq (n+k+1)^T$, $\supp \nu\subseteq A$ and
  \[
    \nu[f]=\mu[f].
  \]
\end{theorem}

\begin{proof}
  By changing $f$ on a $\mu$-nullset and replacing $A$ with a smaller set of full $\mu$-measure, we may assume that $f$ and $A$ are Borel. The case $T=1$ is now a consequence of Tchakaloff's theorem in the form of \cite[Corollary~2]{BayerTeichmann.06} applied to the function $\phi: \X \to \R^{n+k+1}$ given by $\phi(x)=(f(x), x, 1)$.
  Hence, we assume that the theorem holds for some $T\in\N$ and show how to pass to $T+1$. So let $\mu$ be a martingale law on $\X^{T+1}$ and let $A\subseteq\X^{T+1}$ satisfy $\mu(A)=1$.
  Let $\mu_{0}$ be the marginal of $\mu$ on $\X^T$, given by $\mu_0(B):=\mu(B\times \X)$ for $B\in\cB(\X^T)$, and let $\mu_1$ be a Borel-measurable stochastic kernel from $\X^T$ to $(\X,\cB(\X))$
  such that
  \begin{equation}\label{eq:product}
    \mu=\mu_0\otimes \mu_1.
  \end{equation}
  It is easy to see that $\mu_0$ is a martingale law on $\X^T$ and that $\mu_0(A_0)=1$ if $A_0$ is the (universally measurable) canonical projection of $A$ onto $\X^T$.
  On the other hand, it follows from~\eqref{eq:product} that there exists $N\in\cB(\X^T)$ with $\mu_0(N)=0$ such that for all $x\in \X^T \setminus N$, we have  $\int |f(x,x')| \,\mu_1(x;dx')<\infty$ and
  \[
    \mbox{$\mu_1(x)$ is a martingale law on $\X$ satisfying $\mu_1(x;A_x)=1$,}
  \]
  where $A_x\in \cB(\X)$ is the section $A_x=\{x'\in\X:\, (x,x')\in A\}$.

  By the induction hypothesis, there exists a martingale law $\nu_0$ on $\X^T$ such that
  \begin{equation}\label{eq:Support0}
    \#\supp \nu_0\leq (n+k+1)^T,\quad \supp \nu_0\subseteq A_0\setminus N
  \end{equation}
  and
  \[
    \nu_0[g]=\mu_0[g] \quad \mbox{for}\quad g(x):=\int f(x,x') \,\mu_1(x;dx').
  \]

  Fix $x\in \X^T\setminus N$. By applying the case $T=1$ to the function $f(x,\cdot)$ and the measure $\mu_1(x)$, we obtain a martingale law $\nu_1(x)$ on $\X$ such that $\#\supp \nu_{1}\leq n+k+1$, $\supp \nu_{1} \subseteq A_x$ and
  \[
   \int f(x,x')\, \nu_1(x;dx')=\int f(x,x')\, \mu_1(x;dx')\equiv g(x).
  \]
  We may see $x\mapsto \nu_1(x)$ as a kernel and define
  \[
    \nu=\nu_0\otimes \nu_1;
  \]
  this product is well defined as a consequence\footnote{In particular, the finiteness of $\supp \nu_0$ implies that there are no measurability issues; $\nu$ is simply a finite weighted sum.} of~\eqref{eq:Support0}. By construction, we have $\#\supp \nu\leq (n+k+1)^{T+1}$. Moreover, it follows from Fubini's theorem that
  \begin{multline*}
    \mu[f]= \int \bigg[\int f(x,x') \,\mu_1(x;dx')\bigg]\,\mu_0(dx) = \mu_0[g] = \nu_0[g] \\= \int \bigg[\int f(x,x') \,\nu_1(x;dx')\bigg]\,\nu_0(dx) = \nu[f],
  \end{multline*}
  and similarly that $\nu$ is a martingale law satisfying $\nu(A)=1$.
\end{proof}

The preceding theorem entails that even for merely measurable functions~$f$, simple martingales are sufficient to establish martingale inequalities; in particular, this yields an extension of the results from Section~\ref{se:timehom} to general martingales.

\begin{corollary}
  Let $\X=\R^n$ and let $f: \X^{T+1}\to\overline{\R}$ be universally measurable. Then
  \[
    \sup_{\mu\in\cM^T(x_0)} \mu[f] = \sup_M E[f(M_0,\dots, M_T)],
  \]
  where the supremum on the right-hand side is taken over all $n$-dimensional martingales $M_0,\dots,M_T$ with $M_0=x_0$, each on its filtered probability space.
\end{corollary}

\begin{proof}
  It suffices to show that $\sup_{\mu\in\cM^T(x_0)} \mu[f(Z_T)] \geq E[f(M_0,\dots M_T)]$ for any martingale $M$ with $M_0=x_0$. For this, we may assume without loss of generality that $E[f(M_0,\dots M_T)]>-\infty$ and, by monotone convergence, that $f$ is bounded from above. Hence, we may assume that $f$ is real-valued.

  Under these conditions, we have $\mu_M[|f|]<\infty$ for the law $\mu_M$ of $M$. Thus, Theorem~\ref{th:tchakaloff} yields $\mu\in\cM^T(x_0)$ such that $\mu[f]=\mu_M[f]=E[f(M_0,\dots M_T)]$ and the claim follows.
\end{proof}


\newcommand{\dummy}[1]{}


\begin{thebibliography}{10}

\bibitem{AcciaioBeiglbockPenknerSchachermayer.12}
B.~Acciaio, M.~Beiglb{\"o}ck, F.~Penkner, and W.~Schachermayer.
\newblock A model-free version of the fundamental theorem of asset pricing and
  the super-replication theorem.
\newblock {\em To appear in Math. Finance}, 2013.

\bibitem{AcciaioEtAl.12}
B.~Acciaio, M.~Beiglb{\"o}ck, F.~Penkner, W.~Schachermayer, and J.~Temme.
\newblock A trajectorial interpretation of {D}oob's martingale inequalities.
\newblock {\em Ann. Appl. Probab.}, 23(4):1494--1505, 2013.

\bibitem{BayerTeichmann.06}
C.~Bayer and J.~Teichmann.
\newblock The proof of {T}chakaloff's theorem.
\newblock {\em Proc. Amer. Math. Soc.}, 134(10):3035--3040 (electronic), 2006.

\bibitem{BeiglbockHenryLaborderePenkner.11}
M.~Beiglb{\"o}ck, P.~Henry-Labord{\`e}re, and F.~Penkner.
\newblock Model-independent bounds for option prices: a mass transport
  approach.
\newblock {\em Finance Stoch.}, 17(3):477--501, 2013.

\bibitem{BeiglbockSiorpaes.13}
M.~Beiglb{\"o}ck and P.~Siorpaes.
\newblock Pathwise versions of the {B}urkholder--{D}avis--{G}undy inequality.
\newblock {\em Preprint arXiv:1305.6188v1}, 2013.

\bibitem{BouchardNutz.13}
B.~Bouchard and M.~Nutz.
\newblock Arbitrage and duality in nondominated discrete-time models.
\newblock {\em To appear in Ann. Appl. Probab.}, 2013.

\bibitem{BrownHobsonRogers.01}
H.~Brown, D.~Hobson, and L.~C.~G. Rogers.
\newblock Robust hedging of barrier options.
\newblock {\em Math. Finance}, 11(3):285--314, 2001.

\bibitem{Burkholder.81}
D.~L. Burkholder.
\newblock A geometrical characterization of {B}anach spaces in which martingale
  difference sequences are unconditional.
\newblock {\em Ann. Probab.}, 9(6):997--1011, 1981.

\bibitem{Burkholder.84}
D.~L. Burkholder.
\newblock Boundary value problems and sharp inequalities for martingale
  transforms.
\newblock {\em Ann. Probab.}, 12(3):647--702, 1984.

\bibitem{Burkholder.88}
D.~L. Burkholder.
\newblock Sharp inequalities for martingales and stochastic integrals.
\newblock {\em Ast\'erisque}, (157-158):75--94, 1988.
\newblock Colloque Paul L{\'e}vy sur les Processus Stochastiques (Palaiseau,
  1987).

\bibitem{Burkholder.91}
D.~L. Burkholder.
\newblock Explorations in martingale theory and its applications.
\newblock In {\em \'{E}cole d'\'{E}t\'e de {P}robabilit\'es de {S}aint-{F}lour
  {XIX}---1989}, volume 1464 of {\em Lecture Notes in Math.}, pages 1--66.
  Springer, Berlin, 1991.

\bibitem{Burkholder.97}
D.~L. Burkholder.
\newblock Sharp norm comparison of martingale maximal functions and stochastic
  integrals.
\newblock In {\em Proceedings of the {N}orbert {W}iener {C}entenary {C}ongress,
  1994 ({E}ast {L}ansing, {MI}, 1994)}, volume~52 of {\em Proc. Sympos. Appl.
  Math.}, pages 343--358, Providence, RI, 1997. Amer. Math. Soc.

\bibitem{Burkholder.02}
D.~L. Burkholder.
\newblock The best constant in the {D}avis inequality for the expectation of
  the martingale square function.
\newblock {\em Trans. Amer. Math. Soc.}, 354(1):91--105 (electronic), 2002.

\bibitem{CoxObloj.11}
A.~M.~G. Cox and J.~Ob{\l}{\'o}j.
\newblock Robust pricing and hedging of double no-touch options.
\newblock {\em Finance Stoch.}, 15(3):573--605, 2011.

\bibitem{Cox.84}
D.~C. Cox.
\newblock Some sharp martingale inequalities related to {D}oob's inequality.
\newblock In {\em Inequalities in statistics and probability ({L}incoln,
  {N}eb., 1982)}, volume~5 of {\em IMS Lecture Notes Monogr. Ser.}, pages
  78--83. Inst. Math. Statist., Hayward, CA, 1984.

\bibitem{DolinskySoner.12}
Y.~Dolinsky and H.~M. Soner.
\newblock Martingale optimal transport and robust hedging in continuous time.
\newblock {\em To appear in Probab. Theory Related Fields}, 2012.

\bibitem{Hobson.98}
D.~Hobson.
\newblock Robust hedging of the lookback option.
\newblock {\em Finance Stoch.}, 2(4):329--347, 1998.

\bibitem{Hobson.11}
D.~Hobson.
\newblock The {S}korokhod embedding problem and model-independent bounds for
  option prices.
\newblock In {\em Paris-{P}rinceton {L}ectures on {M}athematical {F}inance
  2010}, volume 2003 of {\em Lecture Notes in Math.}, pages 267--318. Springer,
  Berlin, 2011.

\bibitem{Kemperman.68}
J.~H.~B. Kemperman.
\newblock The general moment problem, a geometric approach.
\newblock {\em Ann. Math. Statist}, 39:93--122, 1968.

\bibitem{Obloj.04}
J.~Ob{\l}{\'o}j.
\newblock The {S}korokhod embedding problem and its offspring.
\newblock {\em Probab. Surv.}, 1:321--390, 2004.

\bibitem{OblojYor.06}
J.~Ob{\l}{\'o}j and M.~Yor.
\newblock On local martingale and its supremum: harmonic functions and beyond.
\newblock In {\em From stochastic calculus to mathematical finance}, pages
  517--533. Springer, Berlin, 2006.

\bibitem{Osekowski.12}
A.~Os{\c{e}}kowski.
\newblock {\em Sharp martingale and semimartingale inequalities}, volume~72 of
  {\em Instytut Matematyczny Polskiej Akademii Nauk. Monografie Matematyczne
  (New Series) [Mathematics Institute of the Polish Academy of Sciences.
  Mathematical Monographs (New Series)]}.
\newblock Birkh\"auser/Springer Basel AG, Basel, 2012.

\bibitem{Osekowski.13}
A.~Os{\c{e}}kowski.
\newblock Survey article: {B}ellman function method and sharp inequalities for
  martingales.
\newblock {\em Rocky Mountain J. Math.}, 43(6):1759--1823, 2013.

\bibitem{Tchakaloff.57}
V.~Tchakaloff.
\newblock Formules de cubatures m\'ecaniques \`a coefficients non n\'egatifs.
\newblock {\em Bull. Sci. Math. (2)}, 81:123--134, 1957.

\end{thebibliography}
\end{document}